\documentclass{amsart}

\usepackage{amssymb}
\usepackage{amsfonts}
\usepackage{amsmath}
\usepackage{euscript}
\usepackage{enumerate}
\usepackage{graphics,color}

\usepackage{pdfsync}
\newtheorem{theorem}{Theorem}[section]
\newtheorem{lemma}[theorem]{Lemma}

\newtheorem*{Theorem1'}{Theorem 1'}

\theoremstyle{definition}
\newtheorem{definition}[theorem]{Definition}
\newtheorem{notation}[theorem]{Notation}

\theoremstyle{remark}
\newtheorem{remark}[theorem]{Remark}

\numberwithin{equation}{section}

\definecolor{darkgreen}{rgb}{0,0.5,0.2}
\newcommand \blue{\textcolor{blue}}
\newcommand \green{\textcolor{darkgreen}}

\newcommand \g{{\mathfrak g}}
\newcommand  \s{{\mathfrak s}}
\renewcommand \r{{\mathfrak r}}
\def \n{{\mathfrak n}}
\renewcommand \t{{\mathfrak t}}

\newcommand \gl{{\mathfrak {gl}}}
\renewcommand \sl{{\mathfrak {sl}}}

\newcommand \ad{\text{ad}}

\begin{document}

\title{Jordan-Chevalley decomposition in Lie algebras}

\author{Leandro Cagliero}
\address{CIEM-CONICET, FAMAF-Universidad Nacional de C\'ordoba, C\'ordoba, Argentina.}
\email{cagliero@famaf.unc.edu.ar}
\thanks{The first author was supported in part by CONICET and SECYT-UNC grants.}

\author{Fernando Szechtman}
\address{Department of Mathematics and Statistics, Univeristy of Regina, Canada}
\email{fernando.szechtman@gmail.com}
\thanks{The second author was supported in part by an NSERC discovery grant}





\begin{abstract}
We prove that if $\s$ is a solvable Lie algebra  of matrices over a field of characteristic 0,
and $A\in\s$,
then the semisimple and nilpotent summands of the Jordan-Chevalley decomposition of $A$ belong to $\s$
if and only if there exist $S,N\in\s$, $S$ is semisimple, $N$ is nilpotent (not necessarily $[S,N]=0$)
such that $A=S+N$.
\end{abstract}

\maketitle

\section{Introduction}
All Lie algebras and representations considered in this paper are finite dimensional over a field $\mathbb{F}$
of characteristic 0.
An important question concerning a given representation  $\pi:\g\to\gl(V)$
of a Lie algebra $\g$
is (*) whether $\pi(\g)$ contains the semisimple and nilpotent
parts of the Jordan-Chevalley decomposition (JCD) in $\gl(V)$ of $\pi(x)$ for a given $x\in\g$
(cf.\ Ch.\ VII, \S5 in \cite{Bo7}).
For semisimple Lie algebras, this is true for any representation and this
classic result is a cornerstone of the representation theory of semisimple Lie algebras (see \cite[\S 6.4 and  Ch. VI]{Hu} or 
\cite[\S 9.3 and Ch. 14]{FH}). We are interested in the classification of distinguished classes of indecomposable representations of certain families of non semisimple Lie algebras (see \cite{CS2,CS3}) and an extension of the classical result to more general Lie algebras will prove useful in this endeavour. In a different direction, the recent article \cite{Ki}, studies the existence of a Jordan-Chevalley-Seligman decomposition in prime characteristic.

The question (*) led us to study in \cite{CS} the existence and uniqueness of abstract
JCD's in arbitrary Lie algebras.
Recall that an element $x$ of a Lie algebra $\g$ is said to have an \emph{abstract JCD}
if there exist unique $s,n\in\g$ such that $x=s+n$, $[s,n]=0$ and
given any finite dimensional representation $\pi:\g\to\gl(V)$
the JCD of $\pi(x)$ in $\gl(V)$ is $\pi(x)=\pi(s)+\pi(n)$.
The Lie algebra $\g$ itself is said to have an abstract JCD
if everyone of its elements does.
The main results of \cite{CS} are Theorems 1 and 2 and they respectively state that
\emph{a Lie algebra has an abstract JCD if and only if it is perfect},
and
\emph{an element of a Lie algebra $\g$ has an abstract JCD if and only if it belongs to $[\g,\g]$.}
These theorems, though related to question (*), do not provide a satisfactory answer to it.

The purpose of this note is two-fold: on the one hand we prove Theorem \ref{thm.main}
below which addresses directly question (*) and allows us to derive
from it \cite[Theorems 1 and 2]{CS}.
On the other hand, we recently realized that there is a gap in the original
proof of \cite[Theorems 1 and 2]{CS}, since \cite[Lemma 2.1]{CS}) is not true.
Therefore, we leave \cite[Theorems 1 and 2]{CS} in good standing by giving a correct proof
derived from Theorem \ref{thm.main}.

\begin{theorem}\label{thm.main}
 Let $\s$ be a solvable Lie algebra of matrices, let $A\in\s$ and
 assume that $A=S+N$ with $S,N\in\s$, $S$ semisimple, $N$ nilpotent
 (we are not assuming $[S,N]=0$).
Then the semisimple and nilpotent summands of the JCD of $A$ belong to $\s$.
\end{theorem}

This theorem is a consequence of the following result.

\begin{theorem}\label{thm.main1}
 Let $\mathbb{F}$ be algebraically closed.
 Given a square matrix $A=S+N$ with $S$ semisimple and $N$ nilpotent, let
 $\{S_n\}$ and $\{N_n\}$ be sequences defined inductively by
 \[
  S_0=S\quad\text{ and }\quad  N_0=N,
 \]
 and, if $[S_n,N_n]\ne0$,
 let  $(N_n)_{\lambda_n}$ be
 a non-zero eigenmatrix of $\ad(S_n)$
 corresponding to a non-zero eigenvalue $\lambda_n$ appearing in the  $ad(S_n)$-decomposition
 of $N_n$, and let
 \begin{equation}\label{eq.iteration}
    S_{n+1}=S_n+(N_n)_{\lambda_n}\quad{\text and }\quad  N_{n+1}=N_{n}-(N_n)_{\lambda_n}.
 \end{equation}
(The sequences depend on the choice of the non-zero eigenvalues.)

If $\{S,N\}$ generates a solvable Lie algebra $\s$, then (independently of the choice of the eigenvalues)
\begin{enumerate}[(i)]
 \item $S_n$ is semisimple, $N_n$ is nilpotent and $S_n,N_n\in\s$ for all $n$, and
  \item there is $n_0$ such that $[S_{n_0},N_{n_0}]=0$.
\end{enumerate}
In particular, $A=S_{n_0}+N_{n_0}$ is the Jordan-Chevalley
decomposition of $A$ with both components $S_{n_0},N_{n_0}\in\s$.
Moreover, if $\pi:\s\to\gl(V)$ is a representation such that $\pi(S)$ is semisimple and $\pi(N)$ is nilpotent
then  $\pi(A)=\pi(S_{n_0})+\pi(N_{n_0})$ is the Jordan-Chevalley
decomposition of $\pi(A)$.
\end{theorem}

\section{Jordan-Chevalley decomposition of upper triangular matrices}

This section is devoted to prove Theorem \ref{thm.main1} and thus we assume $\mathbb{F}$ algebraically closed.
Let $\t$ denote the Lie algebra of upper triangular $n\times n$ matrices over $\mathbb{F}$,
let $\t'=[\t,\t]$,
and let $\s$ be a Lie subalgebra of $\t$.

\begin{lemma}\label{lemma.exp}
Let $S,X,N\in\s$ and assume
that $\ad_{\s}(S)(N)=\lambda N$, with $\lambda\in \mathbb{F}$,
  and $\ad_{\s}(S)(X)=\mu X$, with $0\ne\mu\in \mathbb{F}$ (in particular, $X\in\t'$).
  Then
\[
 \exp\left(-\mu^{-1}\ad_{\s}(X)\right)(N) = \sum_{j=0}^{n-1} \frac{(-\mu)^{-j}}{j!}\;\ad_{\s}(X)^j(N)
\]
is an eigenmatrix of $\ad_{\s}(S+X)$ of eigenvalue $\lambda$ and it belongs to $\s$.
In particular, $S$ is semisimple if and only if $S+X$ is semisimple.
\end{lemma}

\begin{proof}
Since $X\in\t'$, we see that $-\mu^{-1}\ad_{\s}(X)$ is a nilpotent derivation of $\s$,
so $\exp\left(-\mu^{-1}\ad_{\s}(X)\right)\in\text{Aut}(\s)$. In particular, $\exp\left(-\mu^{-1}\ad_{\s}(X)\right)(N)\in\s$
and
$$
\begin{aligned}
\left[\exp\left(-\mu^{-1}\ad_{\s}(X)\right)(S),\exp\left(-\mu^{-1}\ad_{\s}(X)\right)(N)\right] &=\exp\left(-\mu^{-1}\ad_{\s}(X)\right)([S,N])\\
&=\lambda \exp\left(-\mu^{-1}\ad_{\s}(X)\right)(N).
\end{aligned}
$$
But $[S,X]=\mu X$ yields $\exp\left(-\mu^{-1}\ad_{\s}(X)\right)(S)=S+X$, so $\exp\left(-\mu^{-1}\ad_{\s}(X)\right)(N)$
is an eigenmatrix of $\ad_{\s}(S+X)$ of eigenvalue $\lambda$. Consequently, if $\ad_{\t}(S)$ is semisimple then
$\exp\left(-\mu^{-1}\ad_{\t}(X)\right)$ transforms a basis of eigenmatrices of $\ad_{\t}(S)$
into a basis of eigenmatrices of $\ad_{\t}(S+X)$.

To complete the proof  it is sufficient to show that
a matrix $A\in\t$ is semisimple if and only if $\ad_{\t}(A)$ is semisimple.
The `only if' part is clear. Conversely,
if $\ad_{\t}(A)$ is semisimple and $A=A_s+A_n$ is the JCD of $A$ then
$A_s,A_n\in\t$ (both are polynomials in $A$) and it follows that
$\ad_{\t}(A)=\ad_{\t}(A_s)+\ad_{\t}(A_n)$ is the JCD of
$\ad_{\t}(A)$. By uniqueness, $\ad_{\t}(A_n)=0$ and this implies $A_n=0$ since
$A_n\in \t'$ and the centralizer of $\t$ in $\t'$ is 0.
\end{proof}

Let $S\in\s$ be semisimple. Let $\Lambda$ be the set of eigenvalues of $\ad_\s(S)$
and for each $\lambda\in\Lambda$ let $\s_\lambda\subset \s$
be the corresponding eigenspace. Given $N\in\s$, let
\[
N=\sum_{\lambda\in\Lambda} N_\lambda,
\]
where each $N_\lambda\in\s_\lambda$. We refer to the above as the $\ad_\s(S)$-decomposition of $N$.

For $k=0,\dots,n-1$, let $\t_k$ be the subspace of $\t$ consisting of those matrices
whose non-zero entries lay only on the diagonal $(i,j)$ such that $j-i=k$.
Given $N\in\t$, let $d_k(N)\in\t_k$ be defined so that $N=\sum_{k=0}^{n-1} d_k(N)$.
 We now introduce a function that will used to measure how close  two matrices are to commuting with each other.
 \begin{definition}\label{def.c}
 Let $S,N\in\s$, with $S$ semisimple, and let
$N=\sum_{\lambda\in\Lambda} N_{\lambda}$ be the decomposition of $N$ as a
sum of eigenmatrices of $\ad_{\s}(S)$.
For $k=0,\dots,n-1$, let
\[
 C_{S,k}(N)=\{\lambda\in\Lambda:\lambda\ne0 \text{ and } d_{k}(N_\lambda)\ne0\},
\]
let $c_{S,k}(N)$ be the number of elements in $C_{S,k}(N)$
($c_{S,0}(N)=0$ since $\lambda\ne0\Rightarrow N_\lambda\in\t'$) and let
  \[
   \gamma_S(N)=\big(c_{S,1}(N),\dots,c_{S,n-1}(N)\big)\in\mathbb{Z}_{\ge0}^{n-1}.
  \]
  It is clear that $c_{S,k}(N)\le \dim\s$ for all $k$ and $[S,N]=0$ if and only if $\gamma_S(N)=(0,\dots,0)$.
  \end{definition}

 \begin{lemma}\label{lemma.c}
 Let $S,X,N\in\s$  with $S$ semisimple and $\ad_{\s}(S)(X)=\mu X$, with $0\ne\mu\in \mathbb{F}$.
 Let $k_0\ge1$ be the lowest $k$ such that $d_k(X)\ne0$ ($\mu\ne0$ implies $X\in\t'$ and hence $k_0\ge1$).
 Then
\[
    C_{S+X,k}(N) =    C_{S,k}(N)
\]
for all $k\le k_0$.
 \end{lemma}

\begin{proof}
 We first point out that it follows from Lemma \ref{lemma.exp} that $S+X$ is semisimple and thus it makes sense
 to consider $C_{S+X,k}(N)$.

Let
\[
 N=\sum_{\lambda\in\Lambda} N_{\lambda},\qquad N_{\lambda}\in\s,
\]
be the $\ad_{\s}(S)$-decomposition of $N$.
Let
\[
\tilde N_{\lambda,0} = \exp(-\mu^{-1} \ad_{\s}(X)) (N_{\lambda})
\]
and, for $j\ge1$, let
$\tilde N_{\lambda,j} = \frac{\mu^{-j}}{j!} \ad_{\s}(X)^j(\tilde N_{\lambda,0})$.

It follows from Lemma \ref{lemma.exp} that $\tilde N_{\lambda,j}$
is an eigenmatrix of $\ad_{\s}(S+X)$ of eigenvalue $\lambda+j\mu$.
Since
\begin{align*}
 N_{\lambda}
 &= \exp(\mu^{-1} \ad_{\s}(X)) (\tilde N_{\lambda,0} ) \\
 &=  \tilde N_{\lambda,0} +\tilde N_{\lambda,1} +\tilde N_{\lambda,2} + \dots
\end{align*}
it follows that
\begin{equation*}\label{eq-enie}
 N=\sum_{\lambda\in\Lambda} \sum_{j\ge0} \tilde N_{\lambda,j}
  =\sum_{\lambda\in\Lambda} \tilde N_{\lambda,0}\;+\;\sum_{\lambda\in\Lambda} \sum_{j\ge1} \tilde N_{\lambda,j}
\end{equation*}
and this leads to the decomposition of $N$ as a
sum of eigenmatrices of $\ad_{\s}(S+X)$ (after adding up those $\tilde N_{\lambda,j}$ with the same eigenvalue).

Let $k\le  k_0$ (recall that $k_0$ is the lowest $k$ such that $d_k(X)\ne0$).
Since $k_0\ge1$, it follows that
\[
 d_k(\tilde N_{\lambda,j})=\begin{cases}
                            d_k(N_{\lambda}),&\text{ if $j=0$};\\
                            0,&\text{ if $j\ge1$}.
                           \end{cases}
\]
This implies $C_{S+X,k}(N) =  C_{S,k}(N)$.
\end{proof}

 \begin{lemma}\label{lemma.norm}
   Let $S,N\in\s$, with $S$ semisimple, and let
$N=\sum_{\lambda\in\Lambda} N_{\lambda}$ be the $\ad_{\s}(S)$-decomposition of $N$.
Assume that there is $\lambda_0\in\Lambda$ with $\lambda_0\ne0$ such that $N_{\lambda_0}\in\s_{\lambda_0}$ is non-zero.
Then
\[
   \gamma_{S+N_{\lambda_0}}(N-N_{\lambda_0}) <   \gamma_{S}(N)
\]
in the lexicographical order. (The pair $(S+N_{\lambda_0},N-N_{\lambda_0})$ is closer to commuting than the pair $(S,N)$.)
 \end{lemma}
\begin{proof}
Let $k_0$ be the lowest $k$ such that $d_k(N_{\lambda_0})\ne0$ ($k_0\ge1$ since $N_{\lambda_0}\in\t'$).
It is clear that
\begin{equation}\label{eq.c1}
c_{S,k}(N-N_{\lambda_0})=
 \begin{cases}
  c_{S,k}(N),&\text{if $k<k_0$}; \\
  c_{S,k_0}(N)-1,&\text{if $k=k_0$};
 \end{cases}
\end{equation}
and thus  $\gamma_{S}(N-N_{\lambda_0}) <   \gamma_{S}(N)$.

 It follows from Lemma \ref{lemma.c} that, for $ k\le k_0$,
 \[
c_{S+N_{\lambda_0},k}(N-N_{\lambda_0})=c_{S,k}(N-N_{\lambda_0}),
\]
and this, combined with \eqref{eq.c1}, implies   $\gamma_{S+N_{\lambda_0}}(N-N_{\lambda_0}) <   \gamma_{S}(N)$
in the lexicographical order.
\end{proof}

We are now in a position to prove Theorem \ref{thm.main1}.

\begin{proof}[Proof of Theorem \ref{thm.main1}]
Since  $\{S,N\}$ generates a solvable Lie algebra $\s$, and $\mathbb{F}$ is algebraically closed,
it follows from Lie's Theorem that we may assume $S,N\in\s\subset\t$ and, since $N$ is nilpotent,  $N\in\t'$.

We will prove (i) by induction. Assume (i) true for $S_{n}$ and $N_{n}$ and let us suppose that
$[S_n,N_n]\ne0$.
Since $\lambda_n\ne0$, we have $(N_n)_{\lambda_n}\in\t'$ and hence  $N_{n+1}$ is nilpotent.
It follows from Lemma \ref{lemma.exp}  that $S_{n+1}$ is semisimple and $S_{n+1},N_{n+1}\in\s$.
This proves (i).

It follows from Lemma \ref{lemma.norm} that
 \[
  \gamma_{S_{n+1}}(N_{n+1})= \gamma_{S_n+(N_n)_{\lambda_n}}(N_n-(N_n)_{\lambda_n}) <   \gamma_{S_n}(N_n)
\]
in the lexicographical order.
This implies that there exists $n_0$ such that  $\gamma_{S_{n_0}}(N_{n_0})=0$ and hence  $[S_{n_0},N_{n_0}]=0$.
This proves (ii) and it is clear $A=S_{n_0}+N_{n_0}$ is the Jordan-Chevalley
decomposition of $A$.

Finally, let $\pi:\s\to\gl(V)$ be a representation such that $\pi(S)=\pi(S_0)$ is semisimple and $\pi(N)=\pi(N_0)$ is nilpotent.
Since $\pi$ is a representation, if
$N_n=\sum_{\lambda\in\Lambda_n} (N_n)_{\lambda}$ is the  $\ad_{\s}(S_n)$-decomposition of $N_n$, then
\[
\pi(N_n)=\sum_{\lambda\in\Lambda_n} \pi((N_n)_{\lambda})
\]
is the $\ad_{\pi(\s)}(\pi(S_n))$-decomposition of $\pi(N_n)$.
Therefore, assuming that $\pi(S_n)$ is semisimple and $\pi(N_n)$ is nilpotent, it follows, just as above, that
 $\pi(S_{n+1})$ is semisimple and $\pi(N_{n+1})$  is nilpotent.
 This implies that  $\pi(A)=\pi(S_{n_0})+\pi(N_{n_0})$ is the Jordan-Chevalley
decomposition of $\pi(A)$.
\end{proof}

\begin{proof}[Proof of Theorem \ref{thm.main}]
 Theorem \ref{thm.main1} shows that
Theorem \ref{thm.main} is true when $\mathbb{F}$ is algebraically closed,
since in this case Lie's Theorem allows us to assume that  $\s\subset\t$.

In general, let $\bar{\mathbb{F}}$ be
an algebraic closure of $\mathbb{F}$.
Suppose $A,S,N\in\s$, where $A=S+N$, $S$ is semisimple and $N$ is nilpotent.
Let $A=S'+N'$ be the JCD of $A$ in $\gl(n,\mathbb{F})$, as ensured in \cite[\S 7.5]{HK}.
The minimal polynomial of $S'$, say $p$, is a product of distinct monic
irreducible polynomials over $\mathbb{F}$ \cite[\S 7.5]{HK}.
Since $\mathbb{F}$ has characteristic 0, it follows that $p$ has distinct roots in $\bar{\mathbb{F}}$,
whence $S'$ is diagonalizable over $\bar{\mathbb{F}}$.
Therefore, $A=S'+N'$ is the JCD of $A$ in $\gl(n,\bar{\mathbb{F}})$. Let $\bar{\s}$ be the $\bar{\mathbb{F}}$-linear span of
$\s$ in $\gl(n,\bar{\mathbb{F}})$. Then $\bar{\s}$ is a solvable subalgebra of $\gl(n,\bar{\mathbb{F}})$.
As the theorem is true over $\bar{\mathbb{F}}$,
we infer $S',N'\in\bar{\s}$.
Thus  $S',N'\in\gl(n,\mathbb{F})\cap\bar{\s}=\s$.
This completes the proof of Theorem \ref{thm.main}.
\end{proof}

\section{Jordan-Chevalley decomposition in a Lie algebra}

\begin{theorem}
An element $x$ of a Lie algebra $\g$ has an abstract JCD
if and only if $x$ belongs to the derived algebra $[\g,\g]$, in which case
the semisimple and nilpotent parts of $x$ also belong to $[\g,\g]$.
\end{theorem}

\begin{proof}[Necessity.] This is clear since any linear
map from $\g$ to $\gl(V)$ such that $\dim\pi(\g)=1$ and $\pi([\g,\g])=0$
is a representation.

\smallskip

\noindent{\it Sufficiency.}
By Ado's theorem we may assume that $\g$ is a Lie algebra of matrices.
 Fix a Levi decomposition
$\g=\g_{s}\ltimes \r$ and let $\n=[\g,\r]$.
We know that $\n$ is nilpotent (see \cite[Lemma C.20]{FH}).
If $x\in [\g,\g]$, then $x=a+r$ for unique
$a\in\g_{s}$ and $r\in \n$.
If $a=a_s+a_n$ is the JCD of the matrix $a$,
since $\g_{s}$ is semisimple, it follows that $a_s,a_n\in \g_{s}=[\g_{s},\g_{s}]$ (see, for instance, \cite[\S6.4]{Hu}).
Let $\s=\mathbb{F} a_s\oplus \mathbb{F} a_n\oplus \n\subset[\g,\g]$.
Since $[\s,\s]\subset\n$ and $\n$ is nilpotent, we obtain that $\s$ is a solvable Lie algebra.
We now apply Theorem \ref{thm.main} to the Lie algebra
$\s$ with $S=a_s$, $N=a_n+r$.
We obtain that if $x=S'+N'$ is the JCD of $x$, then $S',N'\in\s\subset[\g,\g]$.

Finally, let $\pi:\g\to\gl(V)$ be a representation of $\g$.
Since $r\in\n$ it follows that $\pi(r)$ is nilpotent (see \cite[Lemma C.19]{FH} or  \cite[Ch.1, \S5]{Bo1}).
Since $\g_{s}$ is semisimple, $\pi(S)=\pi(a_s)$ is semisimple and $\pi(a_n)$ is nilpotent.
Since $\s$ is solvable it follows from Lie's Theorem that $\pi(N)=\pi(a_n+r)$ is nilpotent.
It follows from Theorem \ref{thm.main1} (applied over a field extension of $\mathbb{F}$) that
$\pi(x)=\pi(S')+\pi(N')$ is the JCD of $\pi(x)$.
\end{proof}



\begin{thebibliography}{FH}

\bibitem[B1]{Bo1}  Bourbaki, N.,
\emph{Lie Groups and Lie Algebras: chapters 1-3}, Springer-Verlag, 1989.

\bibitem[B2]{Bo7}  Bourbaki, N.,
\emph{Lie Groups and Lie Algebras: chapters 7-9}, Springer-Verlag, Berlin, 2005.


\bibitem[CS]{CS} Cagliero, L. and Szechtman F., {\em Jordan-Chevalley decomposition in finite dimensional Lie
algebras}, Proc. Amer. Math. Soc. \textbf{139} (2011) 3909-3913.

\bibitem[CS2]{CS2} Cagliero, L. and Szechtman F., {\em The classification of uniserial $\sl(2)\rtimes V(m)$-modules and a new interpretation of the Racah–Wigner $6j$-symbol}, J. Algebra \textbf{386} (2013) 142-175.

\bibitem[CS3]{CS3} Cagliero, L. and Szechtman F., {\em On the theorem of the primitive element with applications to the representation theory of associative and Lie Algebras}, Canad. Math. Bull. \textbf{57} (2014) 735--748.

\bibitem[FH]{FH} Fulton, W. and Harris, J.,
\emph{Representation theory: A first course}, Graduate Texts in Mathematics, 129.
Readings in Mathematics. Springer-Verlag, New York, 1991.

\bibitem[HK]{HK} Hoffman K. and Kunze, R. \emph{Linear Algebra (2nd Edition)}, Prentice-Hall, New Jersey, 1971.

\bibitem[Hu]{Hu}  Humphreys, J. E.,
\emph{Introduction to Lie Algebras and Representation Theory},
Graduate Texts in Mathematics, 9. Springer-Verlag, New York-Berlin, 1978.

\bibitem[Ki]{Ki}  Kim, K.-T.,
\emph{Criteria for the existence of a Jordan-Chevalley-Seligman decomposition},
J. Algebra \textbf{424} (2015), 376--389.



\end{thebibliography}
\end{document}